\documentclass{amsart}

\usepackage{latexsym}
\usepackage{amsmath,amsfonts,amssymb}
\usepackage{amsthm}
\usepackage{graphicx,xcolor}
\usepackage{overpic}
\usepackage{subcaption}
\usepackage{braket}
\usepackage[hidelinks,bookmarks=true,bookmarksnumbered=true]{hyperref}
\usepackage{multirow,bigdelim}
\usepackage{enumerate}
\usepackage{cleveref}
\usepackage{comment}
\numberwithin{equation}{section}
\usepackage{doi}

\usepackage{aliascnt}

\newtheorem{theorem}{Theorem}[section]

\newaliascnt{lemma}{theorem}
\newtheorem{lemma}[lemma]{Lemma}
\aliascntresetthe{lemma}

\newaliascnt{corollary}{theorem}
\newtheorem{corollary}[corollary]{Corollary}
\aliascntresetthe{corollary}

\newaliascnt{proposition}{theorem}
\newtheorem{proposition}[proposition]{Proposition}
\aliascntresetthe{proposition}

\newaliascnt{claim}{theorem}

\aliascntresetthe{claim}

\newaliascnt{conjecture}{theorem}

\aliascntresetthe{conjecture}

\theoremstyle{definition}

\newaliascnt{definition}{theorem}
\newtheorem{definition}[definition]{Definition}
\aliascntresetthe{definition}

\newaliascnt{problem}{theorem}

\aliascntresetthe{problem}

\newaliascnt{example}{theorem}
\newtheorem{example}[example]{Example}
\aliascntresetthe{example}

\newaliascnt{remark}{theorem}
\newtheorem{remark}[remark]{Remark}
\aliascntresetthe{remark}

\crefname{section}{Section}{Sections}
\crefname{appendix}{Appendix}{Appendices}
\crefname{theorem}{Theorem}{Theorems}
\crefname{lemma}{Lemma}{Lemmas}
\crefname{corollary}{Corollary}	{Corollaries}			
\crefname{proposition}{Proposition}{Propositions}	
\crefname{claim}{Claim}{Claims}
\crefname{conjecture}{Conjecture}{Conjectures}			
\crefname{definition}{Definition}{Definitions}
\crefname{problem}{Problem}{Problems}
\crefname{example}{Example}{Examples}
\crefname{remark}{Remark}{Remarks}
\crefname{figure}{Figure}{Figures}
\crefname{footnote}{Footnote}{Footnotes}
\crefname{equation}{}{}
\crefname{enumi}{}{}
\crefformat{enumi}{(#2#1#3)}
\Crefformat{enumi}{(#2#1#3)}

\newcommand{\QED}{\hfill \ensuremath{\Box}}
\newcommand{\R}{\mathbb{R}}

\newcommand{\ld}{,\ldots,}

\newcommand{\ep}{\varepsilon}
\newcommand{\wt}{\widetilde}

\newcommand{\HD}{\mathrm{HD}}
\newfont{\bg}{cmr9 scaled\magstep2}
\newcommand{\bigzerol}{\smash{\lower1.0ex\hbox{\bg 0}}}
\DeclareMathOperator{\rank}{rank}

\DeclareMathOperator{\img}{Im}

\DeclareMathOperator{\codim}{codim}
\DeclareMathOperator{\id}{id}

\DeclareMathOperator{\diam}{diam}

\title[
Hausdorff measure estimates for crossings
]
{
A transversality theorem for multiple-point crossings under generic linear perturbations with Hausdorff measure estimates
}

\author{Shunsuke Ichiki
}
\address{
Department of Mathematical and Computing Science, 
School of Computing, 
Institute of Science Tokyo, 
Tokyo 152-8552, 
Japan}
\email{ichiki.s.96cb@m.isct.ac.jp}

\begin{document}
\begin{abstract}
We establish a transversality theorem for multiple-point crossings under generic linear perturbations with explicit Hausdorff measure estimates for the exceptional parameter set, and hence explicit upper bounds on its Hausdorff dimension.
This strengthens our earlier result, which showed only that the exceptional parameter set has Lebesgue measure zero.
As applications, we obtain results on normal crossings, injectivity, and embeddings under generic linear perturbations.
The embedding result yields a refinement of Mather's stability theorem for generic projections when the target dimension is more than twice the source dimension, with an explicit upper bound on the Hausdorff dimension of the exceptional set.
\end{abstract}

\maketitle

\section{Introduction}\label{sec:intro}
Transversality theorems are a fundamental tool in singularity theory and in the study of generic properties of smooth mappings.
In 1973, Mather established a remarkable transversality theorem for generic projections as the main theorem of his celebrated paper \cite{Mather1973}.
Let $\mathcal{L}(\R^m,\R^\ell)$ denote the space of all linear mappings from $\R^m$ to $\R^\ell$, which we identify with the Euclidean space $(\R^m)^\ell$ in the obvious way.
Mather's theorem concerns the composition $\pi\circ f:X\to \R^\ell$ of a $C^\infty$ embedding $f$ from a $C^\infty$ manifold $X$ into $\R^m$ and a linear mapping $\pi\in \mathcal{L}(\R^m,\R^\ell)\setminus\Sigma$, where $\Sigma$ is a subset of $\mathcal{L}(\R^m,\R^\ell)$ with Lebesgue measure zero.
The theorem yields important applications concerning compositions of a $C^\infty$ embedding with generic linear mappings (e.g. \cite[Theorems 2 and 3]{Mather1973}).

Motivated by Mather's parametric transversality theorem for linear mappings, Ichiki \cite{Ichiki2018C} established transversality theorems under generic linear perturbations for $1$-jet extensions and for multiple-point crossings.
More precisely, let $f$ be a $C^\infty$ immersion (resp., a $C^\infty$ injection) of a $C^\infty$ manifold $X$ into an open subset $V$ of $\R^m$, and let $g:V\to \R^\ell$ be an arbitrary $C^\infty$ mapping.
Then Ichiki \cite{Ichiki2018C} proved a transversality theorem for the $1$-jet extension (resp., for multiple-point crossings) of the composition $(g+\pi)\circ f:X\to \R^\ell$, where $\pi\in \mathcal{L}(\R^m,\R^\ell)\setminus\Sigma$ and $\Sigma$ is a subset of $\mathcal{L}(\R^m,\R^\ell)$ of Lebesgue measure zero.
Moreover, in \cite{Ichiki2018T}, both of these transversality theorems were extended to the finite differentiability case.

Most parametric transversality theorems, including those discussed above in \cite{Mather1973,Ichiki2018C,Ichiki2018T}, describe the exceptional parameter set only as a set of Lebesgue measure zero.
Geometrically, however, the information that the exceptional set has Lebesgue measure zero is too coarse.
For example, a countable subset of $\R$ and the Cantor set both have Lebesgue measure zero, although their Hausdorff dimensions are very different.
Likewise, in $\R^3$, a curve and a surface both have Lebesgue measure zero, but their Hausdorff dimensions are $1$ and $2$, respectively.
Thus, Hausdorff measure estimates give finer information on the size of exceptional sets, in particular through explicit upper bounds on their Hausdorff dimensions.

A refinement of Thom's parametric transversality theorem in terms of Hausdorff measure was given in \cite{Ichiki2022a}.
Using this refinement, Ichiki \cite{Ichiki2024b} established a transversality theorem for $1$-jet extensions under generic linear perturbations with Hausdorff measure estimates.
In this paper, we establish a transversality theorem for multiple-point crossings under generic linear perturbations with Hausdorff measure estimates (see \cref{thm:main}).
As applications, we obtain results on normal crossings, injectivity, and embeddings.

In the range $\ell>2\dim X$, our embedding result yields a refinement of Mather's stability theorem for generic projections with an explicit upper bound on the Hausdorff dimension of the exceptional set.
In this range, Mather's theorem states that if $f$ is a $C^\infty$ embedding of a compact $C^\infty$ manifold $X$ into $\R^m$, then the set
\begin{align*}
\set{\pi\in \mathcal{L}(\R^m,\R^\ell)\mid \text{$\pi\circ f:X\to \R^\ell$ is not an embedding}}
\end{align*}
has Lebesgue measure zero in $\mathcal{L}(\R^m,\R^\ell)$.
In the same range, our result refines Mather's theorem as follows.
Let $f$ be a $C^r$ embedding of a compact $C^r$ manifold $X$ $(r\ge2)$ into an open subset $V$ of $\R^m$, and let $g:V\to \R^\ell$ be a $C^r$ mapping, where $\ell>2\dim X$.
Then the Hausdorff dimension of the set
\begin{align*}
\set{\pi\in \mathcal{L}(\R^m,\R^\ell)\mid \text{$(g+\pi)\circ f:X\to \R^\ell$ is not an embedding}}
\end{align*}
is bounded above by $m\ell+2\dim X-\ell$.
Moreover, we give an example showing that this bound is sharp in general (see \cref{ex:embedding}).
In the special case where $g=0$ and $r=\infty$, this recovers Mather's theorem in the range $\ell>2\dim X$.
Thus, in this range, our result goes beyond Mather's theorem in three respects.
First, it replaces a generic linear mapping $\pi$ with a generic linear perturbation $g+\pi$ of an arbitrary given mapping $g$.
Second, it weakens the differentiability assumption from $C^\infty$ to $C^r$ $(r\ge 2)$.
Third, it provides explicit Hausdorff measure estimates for the exceptional set, and hence an explicit upper bound on its Hausdorff dimension.

The paper is organized as follows.
In \cref{sec:main}, we state the main theorem.
In \cref{sec:preparation}, we review the definitions of Hausdorff measure and Hausdorff dimension and recall the refinement of Thom's parametric transversality theorem in terms of Hausdorff measure in \cite{Ichiki2022a}.
The proof of the main theorem is given in \cref{sec:main_proof}.
Finally, \cref{sec:app} is devoted to applications of the main theorem, including results on normal crossings and injectivity.
In the range $\ell>2\dim X$, we also obtain an embedding result, which yields the refinement of Mather's theorem discussed above.
\section{Main theorem}\label{sec:main}
In this section, we fix notation for multiple-point crossings and state the main theorem.
Throughout this paper, unless otherwise stated, all manifolds are assumed to be without boundary and second countable.
\begin{definition}\label{def:transversality}
\upshape
Let $X$ and $Y$ be $C^r$ manifolds, and let $Z$ be a $C^r$ submanifold of $Y$ $(r\ge 1)$.
Let $f : X\to Y$ be a $C^1$ mapping. 
\begin{enumerate}[(1)]
\item 
We say that $f:X\to Y$ is \emph{transverse} to $Z$ \emph{at $x\in X$} if either $f(x)\notin Z$, or $f(x)\in Z$ and
\begin{align*}
df_x(T_xX)+T_{f(x)}Z=T_{f(x)}Y.
\end{align*}
\item 
We say that $f:X\to Y$ is \emph{transverse} to $Z$ if $f$ is transverse to $Z$ at every $x\in X$. 
\end{enumerate}
\end{definition}
Let $X$ be a $C^r$ manifold $(r\ge 1)$.
Set 
\begin{align*}
X^{(d)}=\set{(q_1\ld q_d)\in X^d|\text{$q_i\ne q_j$ if $i\ne j$}}.
\end{align*}
Note that $X^{(d)}$ is an open submanifold of $X^d$. 
For a mapping $f:X\to \R^\ell$, let $f^{(d)}:X^{(d)}\to (\R^\ell)^d$ be the mapping given by 
\begin{align*}
f^{(d)}(q_1\ld q_d)=(f(q_1)\ld f(q_d)). 
\end{align*}
Set 
\begin{align*}
    \Delta_d=\set{(y\ld y) \in (\R^\ell)^d \mid y\in \R^\ell}. 
\end{align*}
Then, $\Delta_d$ is a submanifold of $(\R^\ell)^d$ satisfying 
\begin{align*}
\codim \Delta_d =\dim\,(\R^\ell)^d- 
\dim \Delta_d  =\ell (d-1).
\end{align*}

As in \cite{Ichiki2018C}, for an injection $f:X\to \R^m$, define 
\begin{align*}
d_f=\max\Set{d|\forall (q_1\ld q_d)\in X^{(d)}, 
\dim \sum_{i=2}^d \R\overrightarrow{f(q_1)f(q_i)}=d-1}. 
\end{align*}
Since $f$ is injective, we have $2 \leq d_f$. 
Since $f(q_1)\ld f(q_{d_f})$ are points of $\R^m$, we also have $d_f\leq m+1$. 
Thus, we obtain  
\begin{align*}
2\leq d_f \leq m+1. 
\end{align*}

In \cite{Ichiki2018T}, the following result on multiple-point crossings under generic linear perturbations was established. 
It states that the exceptional set has Lebesgue measure zero.
\begin{proposition}[\cite{Ichiki2018T}]\label{thm:main_t}
Let $f:X\to V$ be a $C^r$ injection and $g:V\to \R^\ell$ a $C^r$ mapping, where $r$ is a positive integer, $X$ is a $C^r$ manifold and $V$ is an open subset of $\R^m$. 
Suppose that 
\begin{align}\label{eq:main_t_r}
r>\max\set{d(\dim X-\ell)+\ell|2\le d\le d_f}.
\end{align}
Then, for any integer $d$ satisfying $2\le d\le d_f$, the set 
\begin{align*}
\Sigma_d=\set{\pi\in \mathcal{L}(\R^m,\R^\ell)|\text{$((g+\pi)\circ f)^{(d)}$ is not transverse to $\Delta_d$}}
\end{align*}
has Lebesgue measure zero in $\mathcal{L}(\R^{m},\R^{\ell})$.
\end{proposition}
Note that \cite[Theorem~2]{Ichiki2018C} is \cref{thm:main_t} in the case where all manifolds and mappings are of class $C^\infty$.
Thus, \cref{thm:main_t} extends \cite[Theorem~2]{Ichiki2018C} to the finitely differentiable case.
We now state the main theorem, which refines \cref{thm:main_t} by replacing the Lebesgue measure zero conclusion with Hausdorff measure estimates.
\begin{theorem}\label{thm:main}
Let $f:X\to V$ be a $C^r$ injection and $g:V\to \R^\ell$ a $C^r$ mapping, where $r$ is a positive integer, $X$ is a $C^r$ manifold and $V$ is an open subset of $\R^m$. 
Set
\begin{align*}
\Sigma_d=\set{\pi\in \mathcal{L}(\R^m,\R^\ell)|\text{$((g+\pi)\circ f)^{(d)}$ is not transverse to $\Delta_d$}},
\end{align*}
where $d$ is an integer satisfying $2\leq d \leq d_f$.
Then the following statements hold.
\begin{enumerate}[$(1)$]
\item \label{thm:mainsub1}
Suppose $\dim X^{(d)}-\codim \Delta_d\geq 0$.
Then, for any real number $s$ satisfying 
\begin{align}\label{eq:main1}
s\geq m\ell-1+\frac{\dim X^{(d)}-\codim \Delta_d+1}{r}, 
\end{align}
the set $\Sigma_d$ has $s$-dimensional Hausdorff measure zero in $\mathcal{L}(\R^{m},\R^{\ell})$.
\item \label{thm:mainsub2}
Suppose $\dim X^{(d)}-\codim \Delta_d<0$. 
Then, the following hold:
\begin{enumerate}[\upshape(2a)]
\item\label{thm:mainsub2a}
For any real number $s$ satisfying 
\begin{align}\label{eq:main2}
s>m\ell+\dim X^{(d)}-\codim \Delta_d, 
\end{align}
the set $\Sigma_d$ has $s$-dimensional Hausdorff measure zero in $\mathcal{L}(\R^{m},\R^{\ell})$.
\item\label{thm:mainsub2b}
For any $\pi\in \mathcal{L}(\R^{m},\R^{\ell})\setminus\Sigma_d$, we have $((g+\pi) \circ f)^{(d)}(X^{(d)})\cap \Delta_d=\emptyset$.
\end{enumerate}
\end{enumerate}
\end{theorem}
\begin{remark}\label{rem:main}
We give several remarks on \cref{thm:main}.
\begin{enumerate}[(1)]
\item 
\cref{thm:main} implies \cref{thm:main_t} as follows:
Let $f$ and $g$ be mappings satisfying the assumption of \cref{thm:main_t}.
Let $d$ be any integer satisfying $2\le d\le d_f$.
First, we consider the case $\dim X^{(d)}-\codim \Delta_d\geq 0$.
Since $r$ satisfies \cref{eq:main_t_r}, we have 
\begin{align*}
    r>\dim X^{(d)}-\codim \Delta_d.
\end{align*}
Therefore, we can set $s=m\ell$ in \cref{eq:main1}.
Since $\Sigma_d$ has $m\ell$-dimensional Hausdorff measure zero in $\mathcal{L}(\R^{m},\R^{\ell})$ by \cref{thm:main}~\cref{thm:mainsub1}, $\Sigma_d$ also has Lebesgue measure zero.
In the case $\dim X^{(d)}-\codim \Delta_d<0$, since we can set $s=m\ell$ in \cref{eq:main2}, $\Sigma_d$ has $m\ell$-dimensional Hausdorff measure zero in $\mathcal{L}(\R^{m},\R^{\ell})$ by \cref{thm:main}~\cref{thm:mainsub2a}, which implies that $\Sigma_d$ has Lebesgue measure zero.
\item \label{rem:infty}
In \cref{thm:main}~\cref{thm:mainsub1}, if all manifolds and mappings are of class $C^\infty$, then for any real number $s$ such that $s>m\ell-1$, there exists a sufficiently large positive integer $r$ satisfying \cref{eq:main1}.
Thus, in the $C^\infty$ case, we can replace \cref{eq:main1} by
\begin{align*}
    s>m\ell-1.
\end{align*}
\item 
In \cref{thm:main}~\cref{thm:mainsub2a}, since 
\begin{align*}
m\ell+\dim X^{(d)}-\codim \Delta_d
=m\ell+d\dim X-\ell(d-1)
\geq d\dim X, 
\end{align*}
it is not necessary to assume that $s$ is non-negative.
Note that the last inequality in the above expression follows from the fact that $2\le d$ $(\le d_f)\le m+1$.
\item 
The assumptions \cref{eq:main1} in the $C^\infty$ case (i.e. $s>m\ell-1$) and \cref{eq:main2} cannot be improved in general (see \cref{rem:main_evaluation_1} and \cref{rem:main_evaluation_2}, respectively), which implies that these estimates are sharp in general.
However, it is still an open question whether \cref{eq:main1} is sharp in the $C^r$ case $(r<\infty)$.
\item 
As in \cite{Ichiki2018C}, one has $d_f=3$ in the following situation. 
If $X=S^n$ and $f:S^n\to \R^m$ ($n+1 \leq m$) is the inclusion $f(x)=(x,0\ld 0)$, then $d_f=3$, where $S^n$ is the $n$-dimensional unit sphere centered at the origin $(n\ge 1)$.
Indeed, suppose that there exists a point $(q_1, q_2, q_3)\in (S^n)^{(3)}$ satisfying $\dim \sum_{i=2}^3 \R\overrightarrow{f(q_1)f(q_i)}=1$. 
Since any affine line in $\R^m$ meets $f(S^n)$ in at most two points, this contradicts the assumption. 
Thus, we have $d_f\geq3$. 
On the other hand, since $S^1\times \set{0} \subset f(S^n)$, we get $d_f<4$, where $0=\underbrace{(0\ld 0)}_{(m-2)\text{-tuple}}$. 
Therefore, $d_f=3$. 
\end{enumerate}
\end{remark}

\section{Preliminaries}\label{sec:preparation}
In this section, we recall the notions of Hausdorff measure and Hausdorff dimension and then state the Hausdorff measure version of Thom's parametric transversality theorem, which will be the main tool in the proof of the main theorem.

Let $s$ be an arbitrary non-negative real number.
The \emph{$s$-dimensional Hausdorff outer measure} on $\R^n$ is defined as follows.
Let $B$ be a subset of $\R^n$.
The $0$-dimensional Hausdorff measure of $B$ is the counting measure of $B$; equivalently, it equals the number of points of $B$ if $B$ is finite, and $\infty$ otherwise.
For $s>0$, the $s$-dimensional Hausdorff outer measure of $B$ is defined by 
\begin{align*}
   \lim_{\delta \to 0} \mathcal{H}_{\delta}^s(B),
\end{align*}
where for each $0<\delta \leq \infty$, 
\begin{align*}
\mathcal{H}_{\delta}^s(B)=\inf\Set{\sum_{j=1}^\infty (\diam C_j)^s|B\subset \bigcup_{j=1}^\infty C_j,\  \diam C_j\leq \delta}.
\end{align*}
Here, for a subset $C$ of $\R^n$, we write 
\begin{align*}
\diam C=\sup \set{\|x-y\||x, y\in C}, 
\end{align*}
where $\|z\|$ denotes the Euclidean norm of $z\in \R^n$.
Note that the infimum in $\mathcal{H}_{\delta}^s(B)$ is over all coverings of $B$ by subsets $C_1, C_2,\ldots$ of $\R^n$ satisfying $\diam C_j\leq \delta$ for all positive integers $j$.

Let $s$ be an arbitrary non-negative real number.
Let $N$ be a $C^r$ manifold $(r\geq 1)$ of dimension $n$, and let $\set{(U_\lambda, \varphi_\lambda)}_{\lambda\in \Lambda}$ be a coordinate neighborhood system of $N$.
Then, a subset $\Sigma$ of $N$ has \emph{$s$-dimensional Hausdorff measure zero} in $N$ if for any $\lambda\in \Lambda$, the set $\varphi_\lambda(\Sigma \cap U_\lambda)$ has $s$-dimensional Hausdorff (outer) measure zero in $\R^n$.
Note that this definition does not depend on the choice of a coordinate neighborhood system of $N$.
Moreover, for a subset $\Sigma$ of $N$, define
\begin{align*}
\HD_N(\Sigma)=\inf\Set{s\in[0,\infty)|\text{$\Sigma$ has $s$-dimensional Hausdorff measure zero in $N$}},
\end{align*}
which is called the \emph{Hausdorff dimension of $\Sigma$} in $N$.
\begin{remark}\label{rem:HD_sharpness}
\upshape
Throughout the paper, we state our results mainly in terms of Hausdorff measure, since the corresponding statements for Hausdorff dimension follow immediately from the definition.

More precisely, let $s_0$ be a non-negative real number, and let $\Sigma$ be a subset of $N$.
If $\Sigma$ has $s$-dimensional Hausdorff measure zero in $N$ for any real number $s>s_0$, then
\begin{align}\label{eq:Hausdorff_inequality}
\HD_N(\Sigma)\le s_0.
\end{align}
Moreover, if $\Sigma$ does not have $s_0$-dimensional Hausdorff measure zero in $N$, then equality holds in \cref{eq:Hausdorff_inequality}.
In particular, the sharpness of a Hausdorff measure estimate with threshold $s>s_0$ implies the sharpness of the corresponding Hausdorff dimension estimate \cref{eq:Hausdorff_inequality}.
\end{remark}
Next, we introduce notation and definitions needed to state the main tool in the proof of the main theorem.

\begin{definition}\label{def:delta}
Let $X$ and $Y$ be $C^r$ manifolds, and let $Z$ be a $C^r$ submanifold of $Y$ ($r\geq1$).
Let $f : X\to Y$ be a $C^1$ mapping. 
For any $x\in X$, set  
\begin{align*}
\delta(f,x,Z)&=\left\{ \begin{array}{ll}
0 & \text{if $f(x)\notin Z$}, \\
\dim Y-\dim (df_x(T_xX)+T_{f(x)}Z) & \text{if $f(x)\in Z$}, \\
\end{array} \right.
\\ 
\delta(f,Z)&=\sup \set{\delta(f,x,Z)|x\in X}. 
\end{align*}
\end{definition}
In the $C^\infty$ case, this agrees with the definition in \cite[p.~230]{Mather1973}. 
As in \cite{Bruce2000}, $\delta(f,x,Z)$ measures the extent to which $f$ fails to be transverse to $Z$ at $x$.

\begin{definition}\label{def:W}
Let $X$, $A$, and $Y$ be $C^r$ manifolds ($r\geq1$), and let
$F:X\times A\to Y$ be a $C^1$ mapping.
For any $a\in A$, let $F_a:X\to Y$ be the mapping defined by $F_a(x)=F(x,a)$. 
Let $Z$ be a $C^r$ submanifold of $Y$.
Define  
\begin{align*}
\Sigma(F,Z)&=\set{a\in A|\text{$F_a$ is not transverse to $Z$}},
\\
W(F,Z)&=\set{(x,a)\in X\times A \mid \delta(F_a,x,Z)=\delta(F,(x,a),Z)>0}, 
\\
\delta^*(F,Z)&=\dim X+\dim Z-\dim Y +\delta(F,Z)
\\
&=\dim X-\codim Z+\delta(F,Z), 
\end{align*} 
where $\codim Z=\dim Y-\dim Z$.
\end{definition}
The following lemma is the main tool in the proof of the main theorem.
\begin{lemma}[\cite{Ichiki2022a}]\label{thm:main_lem}
Let $X$, $A$, and $Y$ be $C^r$ manifolds, let $Z$ be a $C^r$ submanifold of $Y$, and let $F:X\times A\to Y$ be a $C^r$ mapping, where $r$ is a positive integer.
Then the following statements hold.
\begin{enumerate}[$(1)$]
\item \label{thm:main_lem1}
Suppose $\delta^*(F,Z)\geq 0$.
Then, for any real number $s$ satisfying 
\begin{align}\label{eq:main_lem}
s\geq \dim A-1+\frac{\delta^*(F,Z)+1}{r}, 
\end{align}
the following $(\alpha)$ and $(\beta)$ are equivalent. 
\begin{enumerate}
\item [$(\alpha)$]
The set $\pi_A(W(F,Z))$ has $s$-dimensional Hausdorff measure zero in $A$, where $\pi_A: X\times A\to A$ is the natural projection.
\item [$(\beta)$] 
The set $\Sigma(F,Z)$ has $s$-dimensional Hausdorff measure zero in $A$.
\end{enumerate}
\item \label{thm:main_lem2}
Suppose $\delta^*(F,Z)<0$.
Then, the following hold:
\begin{enumerate}[\upshape(2a)]
\item \label{thm:main_lem2_w}
We have $W(F,Z)=\emptyset$.
\item \label{thm:main_lem2_h}
For any non-negative real number $s$ satisfying $s>\dim A+\delta^*(F,Z)$, the set $\Sigma(F,Z)$ has $s$-dimensional Hausdorff measure zero in $A$.
\item \label{thm:main_lem2_e}
For any $a\in A\setminus\Sigma(F,Z)$, we have $F_a(X) \cap Z=\emptyset$.
\end{enumerate}
\end{enumerate}
\end{lemma}

\section{Proof of the main theorem}\label{sec:main_proof}
Let $n=\dim X$.
Fix an integer $d$ with $2\le d\le d_f$.
Let $\Gamma : X^{(d)}\times \mathcal{L}(\R^{m},\R^{\ell}) \to (\R^{\ell})^d$ be the $C^r$ mapping given by 
\begin{align*}
\Gamma(q,\pi)=
\left(((g+\pi)\circ f)(q_1)\ld 
((g+\pi)\circ f)(q_d)\right), 
\end{align*}
where $q=(q_1\ld q_d)$.
We apply \cref{thm:main_lem} with $F=\Gamma$ and $Z=\Delta_d$.

First, we show that $\delta(\Gamma,\Delta_d)=0$.
Although the argument is essentially the same as in the proof of \cite[Theorem~2]{Ichiki2018C} or \cite[Theorem~2]{Ichiki2018T}, we include it here for the convenience of the reader.
To prove this, it is enough to show that 
\begin{align}\label{eq:transverse2}
\dim \left(\img d\Gamma_{(\wt{q}, \wt{\pi})}+
T_{\Gamma(\wt{q}, \wt{\pi})}\Delta_d \right)
=d\ell
\end{align}
for any $(\wt{q},\wt{\pi})\in X^{(d)}\times \mathcal{L}(\R^{m},\R^{\ell})$ satisfying $\Gamma(\wt{q},\wt{\pi})\in \Delta_d$.

Let $\set{(V_\lambda ,\varphi_\lambda )}_{\lambda \in \Lambda}$ be a coordinate neighborhood system of $X$. 
There exists a coordinate neighborhood $(V_{\wt{\lambda}_1}\times \cdots \times V_{\wt{\lambda}_d} \times \mathcal{L}(\R^{m},\R^{\ell}), \varphi_{\wt{\lambda}_1}\times \cdots \times \varphi_{\wt{\lambda}_d}\times \id)$ containing $(\wt{q}, \wt{\pi})$ of $X^{(d)}\times \mathcal{L}(\R^{m},\R^{\ell})$, where $\id:\mathcal{L}(\R^{m},\R^{\ell})\to \mathcal{L}(\R^{m},\R^{\ell})$ is the identity mapping, and $\varphi_{\wt{\lambda}_1}\times \cdots \times \varphi_{\wt{\lambda}_d}\times \id: V_{\wt{\lambda}_1}\times \cdots \times V_{\wt{\lambda}_d} \times \mathcal{L}(\R^{m},\R^{\ell})\to \varphi_{\wt{\lambda}_1}(V_{\wt{\lambda}_1})\times \cdots \times\varphi_{\wt{\lambda}_d}(V_{\wt{\lambda}_d})\times \mathcal{L}(\R^{m},\R^{\ell})$ is defined by $(\varphi_{\wt{\lambda}_1}\times \cdots \times \varphi_{\wt{\lambda}_d}\times \id)(q_1\ld q_d,\pi)=(\varphi_{\wt{\lambda}_1}(q_1)\ld \varphi_{\wt{\lambda}_d}(q_d), \id(\pi))$.
Let $x=(x_1\ld x_d)\in (\R^n)^d$ be a local coordinate on $\varphi_{\wt{\lambda}_1}(V_{\wt{\lambda}_1})\times \cdots \times\varphi_{\wt{\lambda}_d}(V_{\wt{\lambda}_d})$.  

Let $(a_{ij})_{1\leq i \leq \ell, 1\leq j \leq m}$ be a representing matrix of a linear mapping $\pi\in \mathcal{L}(\R^m, \R^\ell)$. 
Then, $(g+\pi)\circ f:X\to \R^\ell$ can be written as 
\begin{align}\label{eq:jet2}
(g+\pi)\circ f=
\biggl(g_1\circ f+\sum_{j=1}^{m}a_{1j}f_j 
\ld 
g_\ell \circ f+\sum_{j=1}^{m}a_{\ell j}f_j\biggr),
\end{align}
where $f=(f_1\ld f_m)$, $g=(g_1\ld g_\ell)$ and $(a_{11}\ld  a_{1m}\ld a_{\ell 1}\ld a_{\ell m})\in (\R^m)^\ell$.   
Hence, in these local coordinates, $\Gamma$ is given by
\begin{align*}
&\Gamma \circ \left(\varphi_{\wt{\lambda}_1}\times 
\cdots \times \varphi_{\wt{\lambda}_d}\times \id\right)^{-1}
(x_1\ld x_d,\pi)
\\
&=\left( (g+\pi) \circ f\circ \varphi_{\wt{\lambda}_1}^{-1}(x_1), 
(g+\pi) \circ f\circ \varphi_{\wt{\lambda}_2}^{-1}(x_2)
\ld
(g+\pi) \circ f\circ \varphi_{\wt{\lambda}_d}^{-1}(x_d) 
\right)
\\
&=\left(
g_1\circ \wt{f}(x_1)+\sum_{j=1}^m a_{1j}\wt{f}_j(x_1)
\ld 
g_\ell \circ \wt{f}(x_1)+\sum_{j=1}^m a_{\ell j}\wt{f}_j(x_1), \right.\\
&\hspace{23pt}g_1\circ \wt{f}(x_2)+\sum_{j=1}^m a_{1j}\wt{f}_j(x_2)
\ld 
g_\ell \circ \wt{f}(x_2)+\sum_{j=1}^m a_{\ell j}\wt{f}_j(x_2), 
\\
&\hspace{120pt}\cdots \cdots \cdots , 
\\
&\hspace{20pt}\left. 
g_1\circ \wt{f}(x_d)+\sum_{j=1}^m a_{1j}\wt{f}_j(x_d)
\ld 
g_\ell \circ \wt{f}(x_d)+\sum_{j=1}^m a_{\ell j}\wt{f}_j(x_d)
\right), 
\end{align*}  
where $\wt{f}(x_i)=(\wt{f}_1(x_i)\ld \wt{f}_m(x_i))=(f_1\circ \varphi_{\wt{\lambda}_i}^{-1}(x_i)\ld f_m\circ \varphi_{\wt{\lambda}_i}^{-1}(x_i))$ $(1\leq i\leq d)$. 

For simplicity, set $\wt{x}=(\varphi_{\wt{\lambda}_1}\times \cdots \times \varphi_{\wt{\lambda}_d})(\wt{q})$. 
The Jacobian matrix $J\Gamma_{(\wt{q},\wt{\pi})}$ of $\Gamma$ at $(\wt{q}, \wt{\pi})$ is given on the left below, and the block $B(x_i)$ is defined on the right.
\begin{align*}
J\Gamma_{(\wt{q}, \wt{\pi})}=
\left(
\begin{array}{@{\,\,\,}c@{\,\,\,\,\,}|@{\,\,\,\,\,}c@{\,\,\,}}
\ast & B(x_1) \\
\ast & B(x_2) \\
 \vdots & \vdots    \\ 
\ast & B(x_d) \\
\end{array}
\right)_{(x,\pi)=(\wt{x},\wt{\pi})},\ 
B(x_i)=
\left. 
\left(
\begin{array}{ccccccc}
\mathbf{b}(x_i)& & \bigzerol \\
\bigzerol &     \ddots   &    \\ 
&  &    \mathbf{b}(x_i)
\end{array}
\right)
\right\}
\,\text{$\ell$ rows},
\end{align*}
where $\mathbf{b}(x_i)=(\wt{f}_1(x_i)\ld \wt{f}_m(x_i))$. 
From the description of $T_{\Gamma(\wt{q},\wt{\pi})}\Delta_d$, it suffices to show that the matrix $M$ defined below has rank $d\ell$.
Also, note that there exist invertible matrices $Q_1$ and $Q_2$ satisfying the expression on the right below:
\begin{align*}
M=
\left(
\begin{array}{@{\,\,\,}c@{\,\,\,\,\,}|@{\,\,\,\,\,}c@{\,\,\,}}
E_\ell & B(x_1) \\
E_\ell & B(x_2) \\
 \vdots & \vdots    \\ 
E_\ell & B(x_d) \\
\end{array}
\right)_{x=\wt{x}},\ \ \ \ \ 
Q_1MQ_2&=
\left(
\begin{array}{@{\,\,\,}c@{\,\,\,\,}|@{\,\,\,\,}c@{\,\,\,}}
E_\ell & 0 \\
0 & B(x_2)- B(x_1) \\
 \vdots & \vdots    \\ 
0 & B(x_d)- B(x_1) \\
\end{array}
\right)_{x=\wt{x}},
\end{align*}
where $E_\ell$ is the $\ell\times \ell$ identity matrix.
From $d-1\leq d_f-1$ and the definition of $d_f$, we have  
\begin{align*}
\dim \sum_{i=2}^d\R\overrightarrow{\wt{f}(x_1)\wt{f}(x_i)}=d-1,
\end{align*}
where $x=\wt{x}$.
Hence, $\rank Q_1MQ_2=d\ell$, and thus $\rank M=d\ell$. 
Since we have shown that $\delta(\Gamma,\Delta_d)=0$, we also have 
\begin{align*}
\delta^*(\Gamma, \Delta_d)=\dim X^{(d)}-\codim \Delta_d. 
\end{align*}

We now prove \cref{thm:main}~\cref{thm:mainsub1}. 
Since $\dim X^{(d)}-\codim \Delta_d\geq 0$, we have $\delta^*(\Gamma, \Delta_d)\geq 0$. 
We also obtain  
\begin{align*}
s&\geq m\ell-1+\frac{\dim X^{(d)}-\codim \Delta_d+1}{r}
=m\ell-1+\frac{\delta^*(\Gamma,\Delta_d)+1}{r}. 
\end{align*}
Since $\delta(\Gamma,\Delta_d)=0$, we have $W(\Gamma,\Delta_d)=\emptyset$, and thus $(\alpha)$ of \cref{thm:main_lem}~\cref{thm:main_lem1} holds in this situation.
Therefore, by \cref{thm:main_lem}~\cref{thm:main_lem1}, the set $\Sigma(\Gamma, \Delta_d)$ has $s$-dimensional Hausdorff measure zero in $\mathcal{L}(\R^m,\R^\ell)$.
Since $\Sigma_d=\Sigma(\Gamma,\Delta_d)$, we have \cref{thm:main}~\cref{thm:mainsub1}.

Finally, we prove \cref{thm:main}~\cref{thm:mainsub2}.
Since $\dim X^{(d)}-\codim \Delta_d<0$, we have  $\delta^*(\Gamma, \Delta_d)<0$. 
Since 
\begin{align*}
s&> m\ell+\dim X^{(d)}-\codim \Delta_d
=m\ell+\delta^*(\Gamma, \Delta_d), 
\end{align*}
the set $\Sigma(\Gamma, \Delta_d)$ has $s$-dimensional Hausdorff measure zero by \cref{thm:main_lem}~\cref{thm:main_lem2_h}.
By \cref{thm:main_lem}~\cref{thm:main_lem2_e}, for any $\pi\in \mathcal{L}(\R^m,\R^\ell)\setminus\Sigma(\Gamma,\Delta_d)$, we have $\Gamma_\pi(X^{(d)})\cap\Delta_d=\emptyset$.
Since $\Sigma_d=\Sigma(\Gamma,\Delta_d)$, we have \cref{thm:main}~\cref{thm:mainsub2}.
\QED
\section{Applications}\label{sec:app}
In this section, we apply \cref{thm:main} in the three ranges $\ell \le \dim X$, $\dim X<\ell\le 2\dim X$, and $\ell>2\dim X$.
We thereby obtain results on normal crossings and injectivity.
In the last range, combining the injectivity result proved here with \cite{Ichiki2024b}, we further obtain an embedding result.
This recovers the refinement of Mather's stability theorem for generic projections described in \cref{sec:intro}.
\begin{definition}\label{def:appmain2}
\upshape 
Let $f:X\to \R^\ell$ be a $C^1$ mapping, where $X$ is a $C^r$ manifold $(r\geq 1)$.   
Then, $f$ is called a \emph{mapping with normal crossings} if for any integer $d\geq 2$, the mapping $f^{(d)}:X^{(d)}\to (\R^\ell)^d$ is transverse to $\Delta_d$. 
\end{definition}
We write $|S|$ for the cardinality of $S$.
\subsection{Normal crossings in the range \texorpdfstring{$\ell\le \dim X$}{ell <= dim X}}
\begin{theorem}\label{thm:coronormal_1}
Let $f$ be a $C^r$ injection of an $n$-dimensional $C^r$ manifold $X$ into an open subset $V$ of $\R^m$, and $g:V\to \R^\ell$ a $C^r$ mapping, where $n\ge \ell$ and $r$ is a positive integer.
Let $\Sigma$ be the set of $\pi\in \mathcal{L}(\R^{m},\R^\ell)$ for which there exists an integer $d$ satisfying $2\leq d\leq d_f$ such that $((g+\pi)\circ f)^{(d)}:X^{(d)}\to (\R^\ell)^d$ is not transverse to $\Delta_d$.
Then, for any real number $s$ satisfying
\begin{align}\label{eq:coronormal_1}
s\geq m\ell-1+\frac{d_f(n-\ell)+\ell+1}{r},   
\end{align}
the set $\Sigma$ has $s$-dimensional Hausdorff measure zero in $\mathcal{L}(\R^{m},\R^\ell)$. 
Moreover, if $(g+\pi)\circ f:X\to \R^\ell$ $(\pi\in \mathcal{L}(\R^{m},\R^\ell)\setminus\Sigma)$ has the property that $|((g+\pi)\circ f)^{-1}(y)|\leq d_f$ for every $y\in \R^\ell$, then $(g+\pi)\circ f$ is a $C^r$ mapping with normal crossings. 
\end{theorem}
\begin{remark}\label{rem:coronormal_1}
\upshape
In \cref{thm:coronormal_1}, if all manifolds and mappings are of class $C^\infty$, then we can replace \cref{eq:coronormal_1} by $s>m\ell-1$ by the same argument as in \cref{rem:main}~\cref{rem:infty}.
\end{remark}
\begin{proof}[Proof of \cref{thm:coronormal_1}]
Let $d$ be an integer satisfying $2\leq d \leq d_f$. 
As in \cref{thm:main}, set
\begin{align*}
\Sigma_d=\set{\pi\in \mathcal{L}(\R^m,\R^\ell)|\text{$((g+\pi)\circ f)^{(d)}$ is not transverse to $\Delta_d$}}.
\end{align*}
Since $n\ge \ell$, we obtain 
\begin{align*}
&\dim X^{(d)}-\codim \Delta_d=nd-\ell(d-1)=d(n-\ell)+\ell\ (\geq 0), 
\\ 
&d_f(n-\ell)+\ell \geq 
d(n-\ell)+\ell =\dim X^{(d)}-\codim \Delta_d. 
\end{align*} 
Thus, we also have 
\begin{align*}
s&\geq m\ell-1+\frac{d_f(n-\ell)+\ell+1}{r}
\geq m\ell-1+\frac{\dim X^{(d)}-\codim \Delta_d+1}{r}. 
\end{align*}
Hence, by \cref{thm:main}~\cref{thm:mainsub1}, $\Sigma_d$ has $s$-dimensional Hausdorff measure zero in $\mathcal{L}(\R^{m},\R^\ell)$.
Since $\Sigma=\bigcup_{d=2}^{d_f}\Sigma_d$, the set $\Sigma$ has $s$-dimensional Hausdorff measure zero in $\mathcal{L}(\R^{m},\R^\ell)$.

If $(g+\pi)\circ f:X\to \R^\ell$ ($\pi\in \mathcal{L}(\R^{m},\R^\ell)\setminus\Sigma$) has the property that $|((g+\pi)\circ f)^{-1}(y)|\leq d_f$ for every $y\in \R^\ell$, then for any integer $d$ satisfying $d>d_f$, we have 
\begin{align*}
((g+\pi) \circ f)^{(d)}(X^{(d)})\cap \Delta_d=\emptyset,
\end{align*}
which implies that $(g+\pi)\circ f$ is a mapping with normal crossings. 
\end{proof}
\begin{example}[an example of \cref{thm:coronormal_1}]\label{ex:coronormal_1}
\upshape 
In \cref{thm:coronormal_1}, take $X=V=\R$ and $f(x)=x$.
Let $g:\R\to \R$ be the $C^\infty$ function defined by $g(x)=0$.
Since $d_f=2$ by the definition of $f$, the set $\Sigma$ in \cref{thm:coronormal_1} is given by
\begin{align*}
\Sigma=\set{\pi\in \mathcal{L}(\R,\R)|\text{$(g+\pi)^{(2)}:\R^{(2)}\to \R^2$ is not transverse to $\Delta_2$}}.
\end{align*}
Then, for any $s>0$, the set $\Sigma$ has $s$-dimensional Hausdorff measure zero in $\mathcal{L}(\R,\R)$ by \cref{thm:coronormal_1} and \cref{rem:coronormal_1}.

On the other hand, by a direct calculation, we obtain
\begin{align*}
    \Sigma=\set{\pi\in\mathcal{L}(\R,\R)|\pi=0}.
\end{align*}
Since $\Sigma$ does not have $0$-dimensional Hausdorff measure zero in $\mathcal{L}(\R,\R)$, we cannot improve the assumption $s>0$, which means that in the $C^\infty$ case, \cref{eq:coronormal_1} (i.e. $s>m\ell-1$) is sharp in general.

This example also illustrates an advantage of Hausdorff measure estimates such as those in \cref{thm:coronormal_1}, as opposed to a mere Lebesgue measure statement.
We regard the exceptional set $\Sigma$ as a subset of $\R$ by identifying $\mathcal{L}(\R,\R)$ with $\R$.
Let $K$ be the Cantor set in $\R$.
Then, we have
\begin{align*}
\HD_{\R}(\Sigma) <\HD_{\R}(K)=\frac{\log 2}{\log 3}=0.63\cdots,
\end{align*}
which implies that $\Sigma$ is never equal to $K$.
A Lebesgue measure statement alone does not exclude the possibility that $\Sigma$ is equal to the Cantor set since it also has Lebesgue measure zero in $\R$.
On the other hand, \cref{thm:coronormal_1} in terms of Hausdorff measure guarantees that $\Sigma$ never coincides with a set of positive Hausdorff dimension such as the Cantor set.
\end{example}
As the following remark shows, in the $C^\infty$ case, the assumption \cref{eq:main1} of \cref{thm:main} (i.e. $s>m\ell-1$) is sharp in general. 
\begin{remark}\label{rem:main_evaluation_1}
    In \cref{thm:main}, let $f$ and $g$ be the functions defined in \cref{ex:coronormal_1}.
    Then \(\Sigma_2\) in \cref{thm:main} is given by
\begin{align*}
\Sigma_2=\set{\pi\in \mathcal{L}(\R,\R)|\text{$(g+\pi)^{(2)}:\R^{(2)}\to \R^2$ is not transverse to $\Delta_2$}}.
\end{align*}
    Since $\dim \R^{(2)}-\codim \Delta_2=1$ $(\ge 0)$, for any $s>0$,
the set $\Sigma_2$ has $s$-dimensional Hausdorff measure zero in $\mathcal{L}(\R,\R)$ by \cref{thm:main}~\cref{thm:mainsub1} and \cref{rem:main}~\cref{rem:infty}.
Since $\Sigma_2$ is equal to $\Sigma$ in \cref{ex:coronormal_1}, $\Sigma_2$ does not have $0$-dimensional Hausdorff measure zero in $\mathcal{L}(\R,\R)$.
Thus, we cannot improve the assumption $s>0$. 
\end{remark}
\subsection{Normal crossings in the range \texorpdfstring{$\dim X<\ell\le 2\dim X$}{dim X < ell <= 2 dim X}}
\begin{theorem}\label{thm:coronormal_2}
Let $f$ be a $C^r$ injection of an $n$-dimensional $C^r$ manifold $X$ into an open subset $V$ of $\R^m$, and $g:V\to \R^\ell$ a $C^r$ mapping, where $n<\ell\le 2n$ and $r$ is a positive integer. 
Let $\Sigma$ be the set of $\pi\in \mathcal{L}(\R^{m},\R^\ell)$ for which there exists an integer $d$ satisfying $2\leq d\leq d_f$ such that $((g+\pi)\circ f)^{(d)}:X^{(d)}\to (\R^\ell)^d$ is not transverse to $\Delta_d$.
Then, for any real number $s$ satisfying
\begin{align}\label{eq:coronormal_2}
s\geq m\ell-1+\frac{2n-\ell+1}{r},   
\end{align}
the set $\Sigma$ has $s$-dimensional Hausdorff measure zero in $\mathcal{L}(\R^{m},\R^\ell)$. 
Moreover, if $(g+\pi)\circ f:X\to \R^\ell$ $(\pi\in \mathcal{L}(\R^{m},\R^\ell)\setminus\Sigma)$ has the property that $|((g+\pi)\circ f)^{-1}(y)|\leq d_f$ for every $y\in \R^\ell$, then $(g+\pi)\circ f$ is a $C^r$ mapping with normal crossings. 
\end{theorem}
\begin{remark}\label{rem:coronormal_22}
\upshape
In \cref{thm:coronormal_2}, if all manifolds and mappings are of class $C^\infty$, then we can replace \cref{eq:coronormal_2} by $s>m\ell-1$ by the same argument as in \cref{rem:main}~\cref{rem:infty}.
\end{remark}
\begin{proof}[Proof of \cref{thm:coronormal_2}]
Let $d$ be an integer satisfying $2\leq d \leq d_f$. 
As in \cref{thm:main}, set
\begin{align*}
\Sigma_d=\set{\pi\in \mathcal{L}(\R^m,\R^\ell)|\text{$((g+\pi)\circ f)^{(d)}$ is not transverse to $\Delta_d$}}.
\end{align*}
Since $n<\ell \leq 2n$, we obtain 
\begin{align*}
&\dim X^{(d)}-\codim \Delta_d=nd-\ell(d-1)=d(n-\ell)+\ell\leq 2(n-\ell)+\ell=2n-\ell.
\end{align*} 
First, we consider the case $\dim X^{(d)}-\codim \Delta_d\geq 0$. 
Since 
\begin{align*}
s&\geq m\ell-1+\frac{2n-\ell+1}{r}
\geq m\ell-1+\frac{\dim X^{(d)}-\codim \Delta_d+1}{r}, 
\end{align*}
the set $\Sigma_d$ has $s$-dimensional Hausdorff measure zero in $\mathcal{L}(\R^{m},\R^\ell)$ by \cref{thm:main}~\cref{thm:mainsub1}.
Next, we consider the case $\dim X^{(d)}-\codim \Delta_d<0$. 
Since 
\begin{align*}
s&\geq m\ell-1+\frac{2n-\ell+1}{r}
> m\ell+\dim X^{(d)}-\codim \Delta_d, 
\end{align*}
the set $\Sigma_d$ has $s$-dimensional Hausdorff measure zero in $\mathcal{L}(\R^{m},\R^\ell)$ by \cref{thm:main}~\cref{thm:mainsub2a}.
Since $\Sigma=\bigcup_{d=2}^{d_f}\Sigma_d$, the set $\Sigma$ has $s$-dimensional Hausdorff measure zero in $\mathcal{L}(\R^{m},\R^\ell)$.

The latter assertion can be shown by the same argument as in the proof of \cref{thm:coronormal_1}.  
\end{proof}
\begin{example}[an example of \cref{thm:coronormal_2}]
\upshape
In \cref{thm:coronormal_2}, take $X=V=\R$ and $f(x)=x$.
Let $g:\R\to \R^2$ be the $C^\infty$ mapping defined by $g(x)=(x^2,x^2)$.
Since $d_f=2$ by the definition of $f$, the set $\Sigma$ in \cref{thm:coronormal_2} is given by
\begin{align*}
\Sigma=\set{\pi\in \mathcal{L}(\R,\R^2)|\text{$(g+\pi)^{(2)}:\R^{(2)}\to (\R^2)^2$ is not transverse to $\Delta_2$}}.
\end{align*}
Then, for any $s>1$, the set $\Sigma$ has $s$-dimensional Hausdorff measure zero in $\mathcal{L}(\R,\R^2)$ by \cref{thm:coronormal_2} and \cref{rem:coronormal_22}.
Set 
\begin{align*}
    B=\set{\pi=(\pi_1,\pi_2)\in\mathcal{L}(\R,\R^2)|\pi_1=\pi_2}.
\end{align*}
Then, we can easily obtain $B\subset \Sigma$.
Since $B$ does not have $1$-dimensional Hausdorff measure zero in $\mathcal{L}(\R,\R^2)$, neither does the exceptional set $\Sigma$.
Namely, we cannot improve the assumption $s>1$, which means that in the $C^\infty$ case, the assumption \cref{eq:coronormal_2} (i.e. $s>m\ell-1$) is sharp in general.

Moreover, by the latter assertion of \cref{thm:coronormal_2}, $g+\pi$ is a mapping with normal crossings for any $\pi\in \mathcal{L}(\R,\R^2)\setminus \Sigma$ since $|(g+\pi)^{-1}(y)|\le 2=d_f$ for any $y\in\R^2$.
Here $g$ itself is not a mapping with normal crossings. Nevertheless, \cref{thm:coronormal_2} guarantees that $g+\pi$ is a mapping with normal crossings for generic linear perturbations $\pi$.
\end{example}
\subsection{Injectivity in the range \texorpdfstring{$\ell>2\dim X$}{ell > 2 dim X}}
In \cite{Ichiki2018T}, the following consequence of \cref{thm:main_t} was obtained:
\begin{proposition}[\cite{Ichiki2018T}]\label{thm:injective_t}
Let $f$ be a $C^r$ injection of an $n$-dimensional $C^r$ manifold $X$ into an open subset $V$ of $\R^m$ and $g:V\to \R^\ell$ a $C^r$ mapping, where $\ell>2n$ and $r\geq 1$.
Then, the set 
\begin{align*}
    \Sigma=\set{\pi\in \mathcal{L}(\R^{m},\R^\ell)|\text{$(g+\pi) \circ f:X\to \R^\ell$ is not injective}}
\end{align*}
has Lebesgue measure zero in $\mathcal{L}(\R^{m},\R^\ell)$. 
\end{proposition}
The main theorem also yields the following refinement of \cref{thm:injective_t}.
\begin{theorem}\label{thm:injective}
Let $f$ be a $C^r$ injection of an $n$-dimensional $C^r$ manifold $X$ into an open subset $V$ of $\R^m$ and $g:V\to \R^\ell$ a $C^r$ mapping, where $\ell>2n$ and $r\geq 1$.
Set 
\begin{align*}
    \Sigma=\set{\pi\in \mathcal{L}(\R^{m},\R^\ell)|\text{$(g+\pi) \circ f:X\to \R^\ell$ is not injective}}.
\end{align*}
Then, for any real number $s$ satisfying 
\begin{align}\label{eq:injective}    
s>m\ell+2n-\ell,
\end{align}
the set $\Sigma$ has $s$-dimensional Hausdorff measure zero in $\mathcal{L}(\R^{m},\R^\ell)$. 
\end{theorem}
\begin{proof}[Proof of \cref{thm:injective}]
As in \cref{thm:main}, set
\begin{align*}
\Sigma_2=\set{\pi\in \mathcal{L}(\R^m,\R^\ell)|\text{$((g+\pi)\circ f)^{(2)}$ is not transverse to $\Delta_2$}}.
\end{align*}
Since $2n<\ell$, we obtain 
\begin{align*}
\dim X^{(2)}-\codim \Delta_2=2n-\ell<0. 
\end{align*} 
Thus, we have 
\begin{align*}
s>m\ell+2n-\ell=m\ell+\dim X^{(2)}-\codim \Delta_2.
\end{align*} 
Therefore, the following hold by \cref{thm:main}~\cref{thm:mainsub2}.
\begin{enumerate}[\upshape(a)]
    \item 
    The set $\Sigma_2$ has $s$-dimensional Hausdorff measure zero in $\mathcal{L}(\R^{m},\R^\ell)$.
    \item
    For any $\pi \in \mathcal{L}(\R^{m},\R^{\ell})\setminus\Sigma_2$, we have $((g+\pi) \circ f)^{(2)}(X^{(2)})\cap \Delta_2=\emptyset$. 
\end{enumerate}

By (b), we obtain $\Sigma=\Sigma_2$.
Therefore, (a) completes the proof.
\end{proof}
\begin{example}[an example of \cref{thm:injective}]\label{ex:injective}
\upshape
In \cref{thm:injective}, take $X=V=\R$ and $f(x)=x$.
Let $g:\R\to \R^\ell$ $(\ell \ge 3)$ be the $C^\infty$ mapping defined by 
\begin{align*}
g(x)=(-x^2,-x^3,0\ld 0).
\end{align*}
As in \cref{thm:injective}, set
\begin{align*}
\Sigma=\set{\pi\in \mathcal{L}(\R,\R^\ell)|\text{$g+\pi:\R\to \R^\ell$ is not injective}}.
\end{align*}
Then, for any $s$ satisfying $s>1\cdot \ell+2\cdot 1-\ell=2$, the set $\Sigma$ has $s$-dimensional Hausdorff measure zero in $\mathcal{L}(\R,\R^\ell)$ by \cref{thm:injective}.
Since $\Sigma$ does not have $2$-dimensional Hausdorff measure zero by the following argument, we cannot improve the assumption $s>2$, which means that \cref{eq:injective} is sharp in general.

Now, we show that the set $\Sigma$ does not have $2$-dimensional Hausdorff measure zero in $\mathcal{L}(\R,\R^\ell)$.
Let $\varphi:\R^{(2)}\to \R^2$ be the mapping defined by 
\begin{align*}
\varphi(x,\wt{x})=(x+\wt{x},x^2+x\wt{x}+\wt{x}^2).
\end{align*}
Fix $p:=(x',\wt{x}')\in \R^{(2)}$.
Since $\det J\varphi_p=\wt{x}'-x'\ne 0$, there exist open neighborhoods $U$ $(\subset \R^{(2)})$ of $p$ and $U'$ $(\subset \R^2)$ of $\varphi(p)$ such that $\varphi|_U:U\to U'$ is a diffeomorphism by the inverse function theorem.
Note that for $\pi\in \mathcal{L}(\R,\R^\ell)$, the mapping $g+\pi:\R\to\R^\ell$ can be expressed as follows:
\begin{align*}
    (g+\pi)(x)=\left(-x^2+a_1x,-x^3+a_2x,a_3x\ld a_\ell x\right),
\end{align*}
where $(a_i)_{1\le i\le \ell}$ is the representing matrix of $\pi$.
Now, we regard $\Sigma$ as a subset of $\R^\ell$ by identifying an element $\pi\in\mathcal{L}(\R,\R^\ell)$ with $(a_1\ld a_\ell)\in\R^\ell$.
Since 
\begin{align*}
    B:=U'\times\set{(0\ld 0)}\ (\subset \R^2\times \R^{\ell-2}=\R^\ell)
\end{align*}
    does not have $2$-dimensional Hausdorff measure zero in $\R^\ell$, it suffices to show that $B\subset \Sigma$.

Let $a:=(a_1,a_2,a_3\ld a_\ell)\in B$ be any element.
Then, there exists an element $(x,\wt{x})\in U$ such that $\varphi(x,\wt{x})=(a_1,a_2)$.
We denote the $j$-th component of $g+\pi$ by $(g+\pi)_j$.
Then, we have
\begin{align*}
    (g+\pi)_1(x)-(g+\pi)_1(\wt{x})&=(\wt{x}-x)(x+\wt{x}-a_1)=0,
    \\
    (g+\pi)_2(x)-(g+\pi)_2(\wt{x})&=(\wt{x}-x)(x^2+x\wt{x}+\wt{x}^2-a_2)=0.
\end{align*}
Note that the last equality in each equation above follows from $\varphi(x,\wt{x})=(a_1,a_2)$.
Since $a_3=\cdots =a_\ell=0$, each $(g+\pi)_j$ $(3\le j\le \ell)$ is identically zero.
Therefore, $g+\pi$ is not injective, and thus $a\in \Sigma$.

This example also illustrates the advantage of \cref{thm:injective} over \cref{thm:injective_t} when $\ell\geq 4$.
Since any $3$-dimensional smooth submanifold of $\mathcal{L}(\R,\R^\ell)$ has Lebesgue measure zero, \cref{thm:injective_t} alone does not exclude the possibility that the exceptional set has Hausdorff dimension $3$.
By \cref{thm:injective}, however, we have $\HD_{\mathcal{L}(\R,\R^\ell)}(\Sigma)\leq 2$, so the exceptional set can never coincide with such a set.
\end{example}
As the following remark shows, the assumption \cref{eq:main2} of \cref{thm:main} is sharp in general. 
\begin{remark}\label{rem:main_evaluation_2}
    In \cref{thm:main}, let $f$ and $g$ be the mappings defined in \cref{ex:injective}.
    Then, $\Sigma_2$ in \cref{thm:main} is expressed as follows:
\begin{align*}
\Sigma_2=\set{\pi\in \mathcal{L}(\R,\R^\ell)|\text{$(g+\pi)^{(2)}$ is not transverse to $\Delta_2$}}.
\end{align*}
Since $\dim \R^{(2)}-\codim \Delta_2=2-\ell<0$, we have the following by \cref{thm:main}~\cref{thm:mainsub2}.
\begin{enumerate}
    \item [(a)]
    For any $s>2$, the set $\Sigma_2$ has $s$-dimensional Hausdorff measure zero in $\mathcal{L}(\R,\R^\ell)$.
    \item [(b)]
     For any $\pi \in \mathcal{L}(\R,\R^\ell)\setminus\Sigma_2$, we have $(g+\pi)^{(2)}(\R^{(2)})\cap \Delta_2=\emptyset$. 
\end{enumerate}
Since $\Sigma_2$ is equal to the set $\Sigma$ in \cref{ex:injective} by (b), we cannot improve the assumption $s>2$, which implies that \cref{eq:main2} is sharp in general.
\end{remark}
\subsection{Embeddings and a refinement of Mather's theorem in the range \texorpdfstring{$\ell>2\dim X$}{ell > 2 dim X}}
We recall Mather's stability theorem for generic projections in the range $\ell>2\dim X$.
\begin{proposition}[\cite{Mather1973}]\label{thm:mather_embedding}
Let $f$ be a $C^\infty$ embedding of a compact $C^\infty$ manifold $X$ into $\R^m$ and $\ell$ an integer satisfying $\ell>2\dim X$.
Then, the set
\begin{align*}
    \Sigma=\set{\pi\in \mathcal{L}(\R^{m},\R^\ell)|\text{$\pi\circ f:X\to \R^\ell$ is not an embedding}}
\end{align*}
has Lebesgue measure zero in $\mathcal{L}(\R^{m},\R^\ell)$.
\end{proposition}
We also recall the following result on compositions of immersions with generic linear perturbations:
\begin{theorem}[\cite{Ichiki2024b}]\label{thm:immersion}
Let $f$ be a $C^r$ immersion of an $n$-dimensional $C^r$ manifold $X$ into an open subset $V$ of $\R^{m}$ and $g:V\to \R^{\ell}$ a $C^r$ mapping, where $\ell \geq2n$ and $r\geq 2$. 
Set
\begin{align*}
\Sigma=\set{\pi\in \mathcal{L}(\R^m,\R^\ell)|\text{$(g+\pi) \circ f:X\to \R^\ell$ is not an immersion}}.
\end{align*}
Then, for any real number $s$ satisfying 
\begin{align*}
s>m\ell+(2n-\ell-1), 
\end{align*}
the set $\Sigma$ has $s$-dimensional Hausdorff measure zero in $\mathcal{L}(\R^m,\R^\ell)$.
\end{theorem}

By combining \cref{thm:immersion,thm:injective}, we obtain the following result.
\begin{theorem}\label{thm:injective_immersion}
Let $f$ be a $C^r$ injective immersion of an $n$-dimensional $C^r$ manifold $X$ into an open subset $V$ of $\R^m$ and $g:V\to \R^\ell$ a $C^r$ mapping, where $\ell>2n$ and $r\geq 2$.
Set 
\begin{align*}
    \Sigma=\set{\pi\in \mathcal{L}(\R^{m},\R^\ell)|\text{$(g+\pi)\circ f:X\to \R^\ell$ is not an injective immersion}}.
\end{align*}
Then, for any real number $s$ satisfying 
\begin{align}\label{eq:injective_immersion}
s>m\ell+2n-\ell,
\end{align}
the set $\Sigma$ has $s$-dimensional Hausdorff measure zero in $\mathcal{L}(\R^{m},\R^\ell)$.
\end{theorem}
\begin{remark}
\upshape
\cref{ex:injective} also shows that \cref{eq:injective_immersion} in \cref{thm:injective_immersion} is sharp in general.
\end{remark}
If $X$ is compact, then an injective immersion $(g+\pi)\circ f$ is an embedding (see \cite[p.\ 11]{Golubitsky1973}).
Thus, \cref{thm:injective_immersion} immediately yields the following generalization of \cref{thm:mather_embedding}:
\begin{corollary}\label{thm:embedding}
Let $f$ be a $C^r$ embedding of an $n$-dimensional compact $C^r$ manifold $X$ into an open subset $V$ of $\R^m$ and $g:V\to \R^\ell$ a $C^r$ mapping, where $\ell>2n$ and $r\geq 2$.
Set 
\begin{align*}
    \Sigma=\set{\pi\in \mathcal{L}(\R^{m},\R^\ell)|\text{$(g+\pi)\circ f:X\to \R^\ell$ is not an embedding}}.
\end{align*}
Then, for any real number $s$ satisfying 
\begin{align}\label{eq:embedding}
s>m\ell+2n-\ell,
\end{align}
the set $\Sigma$ has $s$-dimensional Hausdorff measure zero in $\mathcal{L}(\R^{m},\R^\ell)$.
\end{corollary}

As noted in \cref{sec:intro}, the special case $g=0$ and $r=\infty$ of the Lebesgue measure statement in \cref{thm:embedding} recovers \cref{thm:mather_embedding}.
\begin{example}[an example of \cref{thm:embedding}]\label{ex:embedding}
\upshape
Let $X$ be a $1$-dimensional compact $C^\infty$ submanifold of $\R^2$ such that there exists $\ep>0$ satisfying
\begin{align*}
    I:=\set{(x_1,0)\in\R^2|-\ep <x_1<\ep}\subset X.
\end{align*}
Let $f:X\to \R^2$ be the inclusion, and let $g:\R^2\to\R^\ell$ $(\ell\ge 3)$ be the mapping defined by
\begin{align*}
g(x_1,x_2)=(-x_1^2,-x_1^3,0\ld 0).
\end{align*}
As in \cref{thm:embedding}, set
\begin{align*}
\Sigma=\set{\pi\in \mathcal{L}(\R^2,\R^\ell)|\text{$(g+\pi)\circ f:X\to \R^\ell$ is not an embedding}}.
\end{align*}
By \cref{thm:embedding}, for any $s$ satisfying $s>2\ell+2\cdot 1-\ell=\ell+2$, the set $\Sigma$ has $s$-dimensional Hausdorff measure zero in $\mathcal{L}(\R^2,\R^\ell)$.
Since $\Sigma$ does not have $(\ell+2)$-dimensional Hausdorff measure zero by the following argument, we cannot improve the assumption $s>\ell+2$, which means that \cref{eq:embedding} is sharp in general.

We now show that the set $\Sigma$ does not have $(\ell+2)$-dimensional Hausdorff measure zero in $\mathcal{L}(\R^2,\R^\ell)$.
Let $\varphi:I^{(2)}\to \R^2$ be the mapping defined by 
\begin{align*}
\varphi((x_1,0),(\wt{x}_1,0))=(x_1+\wt{x}_1,x_1^2+x_1\wt{x}_1+\wt{x}_1^2).
\end{align*}
Fix $p:=((x'_1,0),(\wt{x}_1',0))\in I^{(2)}$.
Since $\det J\varphi_p=\wt{x}_1'-x_1'\ne 0$, there exist open neighborhoods $U$ $(\subset I^{(2)})$ of $p$ and $U'$ $(\subset \R^2)$ of $\varphi(p)$ such that $\varphi|_U:U\to U'$ is a diffeomorphism by the inverse function theorem.
Note that for $\pi\in \mathcal{L}(\R^2,\R^\ell)$, the mapping $g+\pi:\R^2\to\R^\ell$ can be expressed as follows:
\begin{align*}
    &(g+\pi)(x_1,x_2)
    \\
    &=\left(-x_1^2+a_{11}x_1+a_{12}x_2,-x_1^3+a_{21}x_1+a_{22}x_2,a_{31}x_1+a_{32}x_2\ld a_{\ell 1}x_1+a_{\ell 2}x_2\right),
\end{align*}
where $(a_{ij})_{1\le i\le \ell,1\le j\le 2}$ is the representing matrix of $\pi$.
We identify $\mathcal{L}(\R^2,\R^\ell)$ with the set:
\begin{align*}
\set{(a_{11},a_{21},a_{31}\ld a_{\ell 1},a_{12}\ld a_{\ell 2})|a_{ij}\in \R}.
\end{align*}
Set 
\begin{align*}
    B=\set{(a_{11},a_{21},a_{31}\ld a_{\ell 1},a_{12}\ld a_{\ell 2})\in \R^{2\ell}|(a_{11},a_{21})\in U', a_{31}=\cdots =a_{\ell 1}=0}.
\end{align*}
Since $B$ does not have $(\ell+2)$-dimensional Hausdorff measure zero, it suffices to show that $B\subset \Sigma$.

Let $a:=(a_{11},a_{21}\ld a_{\ell 2})\in B$ be any element.
Since $(a_{11},a_{21})\in U'$, there exists an element $((x_1,0),(\wt{x}_1,0))\in U$ such that 
\begin{align}\label{eq:phi}
    \varphi((x_1,0),(\wt{x}_1,0))=(a_{11},a_{21}).
\end{align}
We denote the $j$-th component of $g+\pi$ by $(g+\pi)_j$.
Then, we obtain
\begin{align*}
    (g+\pi)_1(f(x_1,0))-(g+\pi)_1(f(\wt{x}_1,0))&=(\wt{x}_1-x_1)(x_1+\wt{x}_1-a_{11})=0,
    \\
    (g+\pi)_2(f(x_1,0))-(g+\pi)_2(f(\wt{x}_1,0))&=(\wt{x}_1-x_1)(x_1^2+x_1\wt{x}_1+\wt{x}_1^2-a_{21})=0.
\end{align*}
Note that the last equality in each equation above follows from \cref{eq:phi}.
Moreover, since $a_{31}=\cdots =a_{\ell 1}=0$, for each $j$ $(3\le j\le \ell)$ we have  
\begin{align*}
(g+\pi)_j(f(x_1,0))=(g+\pi)_j(f(\wt{x}_1,0))=0.
\end{align*}
Therefore, $(g+\pi)\circ f$ is not injective, and thus $a\in \Sigma$.
\end{example}

\section*{Acknowledgements}
The author was supported by JSPS KAKENHI Grant Numbers JP26K06776 and JP21K13786. 
This work was supported by the Research Institute for Mathematical Sciences, an International Joint Usage/Research Center located in Kyoto University.
\bibliographystyle{plain}
\bibliography{main}
\end{document}